\documentclass[preprint]{elsarticle}
\usepackage{amssymb,amsmath,amsthm,latexsym}
\usepackage{psfrag}
\usepackage{graphicx}
\usepackage{url}
\usepackage{psfrag}
\usepackage{graphpap}
\usepackage{pstricks}
\usepackage{pst-node}
\usepackage{relsize}
\usepackage[left=2cm,right=2cm,top=3cm,bottom=3cm,a4paper]{geometry}

\theoremstyle{plain}
\newtheorem{Thm}{Theorem}
\newtheorem{Lem}[Thm]{Lemma}
\newtheorem{Prop}[Thm]{Proposition}

\newtheorem{Rem}[Thm]{Remark}

\newcommand{\cC}{\ensuremath{\mathcal{C}}}

\newcommand{\cG}{\ensuremath{\mathcal{G}}}
\newcommand{\cH}{\ensuremath{\mathcal{H}}}

\newcommand{\bi}[2]{{#1 \choose #2}}

\journal{arXiv}

\begin{document}

\begin{frontmatter}

\title{The competition number of a graph and the dimension of its hole space}

\author[label1,label5]{Suh-Ryung KIM}
\author[label2]{Jung Yeun LEE}
\author[label3]{Boram PARK}
\author[label4,label6]{Yoshio SANO \corref{cor1}}

\address[label1]{Department of Mathematics Education,
Seoul National University, Seoul 151-742, Korea.}
\address[label2]{National Institute for Mathematical Sciences,
Daejeon 305-390, Korea.}
\address[label3]{Department of Economics,
Seoul National University, Seoul 151-742, Korea.}
\address[label4]{National Institute of Informatics, 
Tokyo 101-8430, Japan}
\fntext[label5]{This research was supported by Basic Science Research
Program through the National Research Foundation of Korea (NRF) funded
by the Ministry of Education, Science and Technology (700-20100058).}
\fntext[label6]{The author was supported 
by JSPS Research Fellowships for Young Scientists.}

\cortext[cor1]{Corresponding author:
sano@nii.ac.jp ; y.sano.math@gmail.com}

\begin{abstract}
The competition graph of a digraph $D$ is a (simple undirected) graph
which has the same vertex set as $D$
and has an edge between $x$ and $y$ if and only if
there exists a vertex $v$ in $D$ such that
$(x,v)$ and $(y,v)$ are arcs of $D$.
For any graph $G$, $G$ together with sufficiently many isolated vertices
is the competition graph of some acyclic digraph.
The competition number $k(G)$ of $G$
is the smallest number of such isolated vertices.
In general, it is hard to compute the competition number $k(G)$
for a graph $G$ and it has been one of important research problems
in the study of competition graphs to characterize a graph
by its competition number. Recently, the relationship
between the competition number and the number of holes of a graph
is being studied.
A hole of a graph is a cycle of length at least $4$ as an induced subgraph.
In this paper, we conjecture that the dimension of the
hole space of a graph is no smaller than the competition number of
the graph. We verify this conjecture for various kinds of graphs and
show that our conjectured inequality is indeed an equality for
connected triangle-free graphs.
\end{abstract}

\begin{keyword}
competition graph; competition number; cycle space; hole; hole space

\MSC[2010] 05C75, 05C20

\end{keyword}

\end{frontmatter}

%%%%%%%%%%%%%%%%%%%%%%%%%%%%%%%%%%%%%%%%%%%%%%%%%%%%%%%%%%%%%%%%%%%%%%%
\section{Introduction}
%%%%%%%%%%%%%%%%%%%%%%%%%%%%%%%%%%%%%%%%%%%%%%%%%%%%%%%%%%%%%%%%%%%%%%%%%

Suppose $D$ is an acyclic digraph.
The {\it competition graph} of $D$, denoted by $C(D)$,
is the (simple undirected) graph which
has the same vertex set as $D$
and has an edge between two distinct vertices $x$ and $y$
if and only if
there exists a vertex $v$ in $D$ such that $(x,v)$ and $(y,v)$
are arcs of $D$.
For any graph $G$, $G$ together with sufficiently many isolated
vertices is the competition graph of an acyclic digraph.
From this observation, Roberts \cite{cn}
defined the {\em competition number} $k(G)$ of
a graph $G$ to be the smallest number $k$ such that $G$
together with $k$ isolated vertices is
the competition graph of an acyclic digraph.

The notion of competition graph was introduced by Cohen~\cite{co}
as a means of determining the smallest dimension
of ecological phase space.
Since then, various variations have been defined and studied
by many authors (see \cite{kimsu,lu} for surveys).
Besides an application to ecology, the concept of competition graph
can be applied to a variety of fields, as summarized in \cite{RayRob}.

Roberts~\cite{cn} observed that characterization of competition
graph is equivalent to computation of competition number.
It does not seem to be easy in general to compute $k(G)$
for a given graph $G$,
as Opsut~\cite{op} showed that the computation of the
competition number of a graph is an NP-hard problem
(see \cite{kimsu,KF} for graphs whose competition numbers are known).
It has been one of important research problems in the study of
competition graphs to characterize a graph by its competition number.
From this point of view, Cho and Kim~\cite{ck} and Kim~\cite{compone}
studied the relationship between the competition number and the number
of holes of a graph.
A cycle of length at least 4 of a graph
as an induced subgraph is called a {\em hole} of the graph
and a graph without holes is called a {\em chordal graph}.
For a graph $G$, we denote the set of all holes of $G$ by $H(G)$
and denote the number of holes of $G$ by $h(G)$. The chordal graphs 
(Roberts \cite{cn}) and
the family of graphs with exactly one hole (Cho and Kim \cite{ck})
satisfy the inequality $k(G) \le h(G)+1$. Based on these facts,
Kim~\cite{compone} conjectured that $k(G) \le h(G)+1$ for a graph $G$
and the following are the families of graphs which were found
in efforts to answer the conjecture.
\begin{itemize}
\item
$\cG_1=\{G \mid h(G)=2\}$
(Lee, Kim, Kim, and Sano~\cite{twoholes}; Li and Chang~\cite{LC});

\item
$\cG_2=\{G \mid \mbox{all the holes of $G$ are independent}\}$
(Li and Chang \cite{indep});

\item
$\cG_3=\{G \mid \mbox{any two distinct holes of $G$
are mutually edge-disjoint}\}$
(Kim, Lee, and Sano \cite{KLS});

\item
$\cG_4=\{G \mid (\forall C \in H(G))(\exists e_C \in E(C))
[e_C \mbox{ belongs to no other induced cycle of }G]\}$
(Kamibeppu \cite{Kamibeppu});
\end{itemize}
where a hole $C$ of a graph $G$ is called {\it independent}
if, for any hole $C'$ of $G$ other than $C$, 
the following two conditions hold: 
\begin{itemize}
\item[-]
$|V(C) \cap V(C')| \leq 2$.
\item[-]
If $|V(C) \cap V(C')| = 2$,
then $|E(C) \cap E(C')|=1$ and $|V(C)| \geq 5$.
\end{itemize}

Lee~{\it et al.}~\cite{smallcomp} 
also studied on graphs
having many holes but with small competition number.
In this paper, we propose the dimension of its hole space
as an upper bound for the competition number of a graph
and show that the inequality holds for various families of graphs,
including the graph families given above except those graphs having holes
sharing edges.
As a matter of fact, this bound equals the competition number
for several interesting classes of graphs
including the family of a nontrivial connected triangle-free graphs.

%%%%%%%%%%%%%%%%%%%%%%%%%%%%%%%%%%%%%%%%%%%%%%%%%%%%%%%%%%%%%%%%%%%%%%%%%%
\section{The hole space of a graph}
%%%%%%%%%%%%%%%%%%%%%%%%%%%%%%%%%%%%%%%%%%%%%%%%%%%%%%%%%%%%%%%%%%%%%%%%%%%

Let $\mathbb{F}_2$ be the finite field of order $2$.
We take a graph $G$ and let $\mathbb{F}_2^{E(G)}$ denote
the set of maps from $E(G)$ to $\mathbb{F}_2$.
For a cycle $C$ of $G$, we define a map
$\chi_C: E(G) \rightarrow \mathbb{F}_2$ by
\[
\chi_C(e) := \left\{
\begin{array}{cl}
1 & \mbox{ if } e \in E(C); \\
0 & \mbox{ otherwise.}
\end{array}
\right.
\]
(We may regard $\chi_C$ as a vector in $\mathbb{F}_2^{|E(G)|}$
once an edge labeling is given.)
Then
\[
\cC(G):=\mbox{Span}\{\chi_C \in \mathbb{F}_2^{E(G)} \mid C
\mbox{ is a cycle of } G \}
\]
is the cycle space of $G$.
For every connected graph $G$,
$\dim \cC(G)=|E(G)|-|V(G)|+1$ (see Theorem 1.9.6 of \cite{dstl}).
Now we set
\[
\cH(G):=\mbox{Span}\{\chi_C \in \mathbb{F}_2^{E(G)} \mid C \in H(G) \}.
\]
Since a hole is a cycle, $\cH(G)$ is a subspace of
the cycle space $\cC(G)$ of $G$.
We call $\cH(G)$ the {\em hole space} of $G$
and $\dim \cH(G)$ the {\em hole dimension} of $G$.

\begin{figure}[h]
\center{\scalebox{1}
{
\begin{pspicture}(0,-1.4129688)(4.1428123,1.4129688)
\psline[linewidth=0.02cm,fillcolor=black,dotsize=0.07055555cm 3.0]{**-}(0.7595312,1.0745312)(0.7595312,-0.40546882)
\psline[linewidth=0.02cm,fillcolor=black,dotsize=0.07055555cm 3.0]{**-}(3.3395312,-0.46546882)(3.3395312,1.0545311)
\psline[linewidth=0.02cm,fillcolor=black,dotsize=0.07055555cm 3.0]{**-}(3.419531,1.0345311)(0.7795312,1.0345311)
\psline[linewidth=0.02cm,fillcolor=black,dotsize=0.07055555cm 3.0]{**-}(0.6995312,-0.3854688)(3.3295312,-0.3954688)
\psdots[dotsize=0.13](2.0395312,1.0145313)
\psdots[dotsize=0.13](2.0395312,-0.3854688)
\usefont{T1}{ptm}{m}{n}
\rput(2.0523436,-1.2354687){$G$}
\psdots[dotsize=0.13](0.7595311,0.3145312)
\psdots[dotsize=0.13](3.3395312,0.3145312)
\psline[linewidth=0.02cm](2.0395312,-0.38546884)(0.75999993,0.31296876)
\psline[linewidth=0.02cm](2.0595312,1.0345311)(0.7799999,0.33296874)
\psline[linewidth=0.02cm](3.34,0.31296876)(2.0195312,1.0345311)
\usefont{T1}{ptm}{m}{n}
\rput(0.6523437,1.2245313){$v_1$}
\usefont{T1}{ptm}{m}{n}
\rput(0.6523437,-0.5954688){$v_2$}
\usefont{T1}{ptm}{m}{n}
\rput(2.0323436,-0.5954688){$v_7$}
\usefont{T1}{ptm}{m}{n}
\rput(3.4923437,-0.5954688){$v_3$}
\usefont{T1}{ptm}{m}{n}
\rput(3.4923437,1.2245313){$v_4$}
\usefont{T1}{ptm}{m}{n}
\rput(1.9923438,1.2245313){$v_5$}
\usefont{T1}{ptm}{m}{n}
\rput(0.41234374,0.3045312){$v_6$}
\usefont{T1}{ptm}{m}{n}
\rput(3.6523438,0.3045312){$v_8$}
\psdots[dotsize=0.12](2.04,0.33296874)
\usefont{T1}{ptm}{m}{n}
\rput(2.3323438,0.20453125){$v_9$}
\psline[linewidth=0.02cm](2.0395312,1.0345311)(2.0395312,-0.3654688)
\psline[linewidth=0.02cm](3.34,0.31296876)(2.04,-0.38703126)
\end{pspicture}
}
}
\caption{A graph with more than two triangles and more than two holes.}
\label{motivation}
\end{figure}
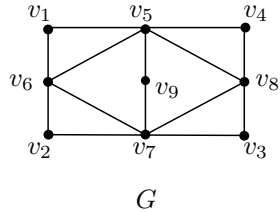

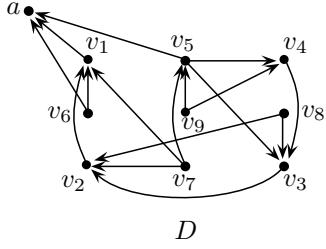
\begin{figure}[h]
\center{
\scalebox{1}
{
\begin{pspicture}(0,-1.6529688)(4.662812,1.6529688)
\psdots[dotsize=0.13](2.4795313,0.77453125)
\psdots[dotsize=0.13](2.4795313,-0.6254688)
\usefont{T1}{ptm}{m}{n}
\rput(2.4923437,-1.4754688){$D$}
\psdots[dotsize=0.13](1.1995311,0.0745312)
\psdots[dotsize=0.13](3.7795312,0.0745312)
\usefont{T1}{ptm}{m}{n}
\rput(1.3323437,0.9845312){$v_1$}
\usefont{T1}{ptm}{m}{n}
\rput(1.0123436,-0.8554688){$v_2$}
\usefont{T1}{ptm}{m}{n}
\rput(2.4723437,-0.8354688){$v_7$}
\usefont{T1}{ptm}{m}{n}
\rput(3.9323437,-0.8354688){$v_3$}
\usefont{T1}{ptm}{m}{n}
\rput(3.9323437,0.9845312){$v_4$}
\usefont{T1}{ptm}{m}{n}
\rput(2.4523437,1.0045313){$v_5$}
\usefont{T1}{ptm}{m}{n}
\rput(0.83234376,0.0645312){$v_6$}
\usefont{T1}{ptm}{m}{n}
\rput(4.1723437,0.0845312){$v_8$}
\psdots[dotsize=0.12](2.48,0.09296875)
\usefont{T1}{ptm}{m}{n}
\rput(2.5923438,-0.07546875){$v_9$}
\psdots[dotsize=0.13](0.4195311,1.4345312)
\psline[linewidth=0.02cm,arrowsize=0.093cm 2.5,arrowlength=1.4,arrowinset=0.4]{->}(1.18,0.79296875)(0.49999994,1.3529687)
\psline[linewidth=0.02cm,arrowsize=0.093cm 2.5,arrowlength=1.4,arrowinset=0.4]{->}(1.1999999,0.09296875)(0.43999994,1.3329687)
\psline[linewidth=0.02cm,arrowsize=0.093cm 2.5,arrowlength=1.4,arrowinset=0.4]{->}(2.48,0.79296875)(0.49999994,1.4129688)
\psdots[dotsize=0.12](1.1999999,0.79296875)
\psdots[dotsize=0.12](3.78,0.79296875)
\psdots[dotsize=0.12](1.18,-0.6070312)
\psline[linewidth=0.02cm,arrowsize=0.093cm 2.5,arrowlength=1.4,arrowinset=0.4]{->}(2.46,-0.6070312)(1.2399999,0.6929687)
\psline[linewidth=0.02cm,arrowsize=0.093cm 2.5,arrowlength=1.4,arrowinset=0.4]{->}(1.1999999,0.07296875)(1.1999999,0.71296877)
\psbezier[linewidth=0.02,arrowsize=0.093cm 2.5,arrowlength=1.4,arrowinset=0.4]{->}(1.18,-0.58703125)(1.02,-0.22703125)(0.93999994,-0.00703125)(1.14,0.71296877)
\psbezier[linewidth=0.02,arrowsize=0.093cm 2.5,arrowlength=1.4,arrowinset=0.4]{->}(3.76,-0.64703125)(3.3567998,-1.1182812)(1.8199999,-1.1870313)(1.2399999,-0.66515625)
\psline[linewidth=0.02cm,arrowsize=0.093cm 2.5,arrowlength=1.4,arrowinset=0.4]{->}(2.46,-0.62703127)(1.2399999,-0.62703127)
\psline[linewidth=0.02cm,arrowsize=0.093cm 2.5,arrowlength=1.4,arrowinset=0.4]{->}(3.78,0.07296875)(1.28,-0.5470312)
\psbezier[linewidth=0.02,arrowsize=0.093cm 2.5,arrowlength=1.4,arrowinset=0.4]{->}(3.8,0.77296877)(3.96,0.41296875)(4.04,0.19296876)(3.84,-0.52703124)
\psdots[dotsize=0.13](3.7795312,-0.62546873)
\psline[linewidth=0.02cm,arrowsize=0.093cm 2.5,arrowlength=1.4,arrowinset=0.4]{->}(2.46,0.79296875)(3.72,-0.5470312)
\psline[linewidth=0.02cm,arrowsize=0.093cm 2.5,arrowlength=1.4,arrowinset=0.4]{->}(3.76,0.07296875)(3.76,-0.52703124)
\psline[linewidth=0.02cm,arrowsize=0.093cm 2.5,arrowlength=1.4,arrowinset=0.4]{->}(2.46,0.09296875)(3.72,0.73296875)
\psline[linewidth=0.02cm,arrowsize=0.093cm 2.5,arrowlength=1.4,arrowinset=0.4]{->}(2.48,0.77296877)(3.6799998,0.77296877)
\psline[linewidth=0.02cm,arrowsize=0.093cm 2.5,arrowlength=1.4,arrowinset=0.4]{->}(2.48,0.07296875)(2.48,0.71296877)
\psbezier[linewidth=0.02,arrowsize=0.093cm 2.5,arrowlength=1.4,arrowinset=0.4]{->}(2.48,-0.6070312)(2.32,-0.23595433)(2.24,-0.009185096)(2.44,0.73296875)
\usefont{T1}{ptm}{m}{n}
\rput(0.22234374,1.4645313){$a$}
\end{pspicture}
}
}
\caption{As $i < j$ whenever $(v_j,v_i) \in A(D)$, $D$ is acylic.
It can be checked that $C(D)=G \cup \{a\}$ for $G$
in Figure~\ref{motivation}.}
\label{digraph}
\end{figure}

For an illustration, consider the graph $G$ given in Figure~\ref{motivation}.
There are exactly three holes
$C_1:=v_5v_6v_7v_9v_5$, $C_2:=v_5v_9v_7v_8v_5$,
and $C_3:=v_5v_6v_7v_8v_5$ in $G$.
We let
$e_1=v_5v_6$, $e_2=v_6v_7$, $e_3=v_7v_9$, $e_4=v_5v_9$,
$e_5=v_7v_8$, $e_6=v_5v_8$, $e_7=v_1v_5$, $e_8=v_1v_6$,
$e_{9}=v_2v_6$, $e_{10}=v_2v_7$, $e_{11}=v_3v_7$, $e_{12}=v_3v_8$,
$e_{13}=v_4v_8$, $e_{14}=v_4v_5$.
Then
\begin{align*}
\chi_{C_1} &= (1,1,1,1,0,0,0,0,0,0,0,0,0,0); \\
\chi_{C_2} &= (0,0,1,1,1,1,0,0,0,0,0,0,0,0); \\
\chi_{C_3} &= (1,1,0,0,1,1,0,0,0,0,0,0,0,0).
\end{align*}
Since $\chi_{C_1}=\chi_{C_2}+\chi_{C_3}$ and $\chi_{C_2}$, $\chi_{C_3}$
are linearly independent, the hole dimension of $G$ is $2$.

Note that $\dim \cH(G) \leq |H(G)|=h(G)$.
Thus any graph $G$ satisfying $k(G) \leq \dim \cH(G)+1$ satisfies
 Kim's conjecture.
The competition number of the graph $G$ in Figure~\ref{motivation}
is $1$ and the digraph $D$ in Figure~\ref{digraph} is an acyclic digraph
whose competition graph is $G \cup \{a\}$.
Thus $k(G) \leq \dim \cH(G) +1$.
In $D$, $a$ covers edges $v_1v_5$, $v_1v_6$, $v_5v_6$
(`a vertex covers an edge' means that the vertex is a common out-neighbor
of the ends of the edge);
$v_1$ covers edges $v_2v_6$, $v_2v_7$, $v_6v_7$;
$v_2$ covers edges $v_3v_7$, $v_3v_8$, $v_7v_8$;
$v_3$ covers edges $v_4v_5$, $v_4v_8$, $v_5v_8$.
Thus the vertices $a$, $v_1$, $v_2$, $v_3$ cover
the edges of the cycle $C_3$.
Similarly we may check that the vertices $v_2$, $v_3$, $v_4$, $v_5$
cover the edges of the cycle $C_2$.
This observation tells us that the
assigned out-neighbors which cover the edges of $C_2$ and $C_3$
also cover the edges of the hole $C_1$.
It motivates us to introduce the notion of hole dimension
by a desire to find a sharp upper bound
for the competition number of a graph. 
As a matter of fact, 
we believe that any graph $G$ satisfies the following inequality: 
\[k(G) \leq \dim \cH(G)+1.\]

%%%%%%%%%%%%%%%%%%%%%%%%%%%%%%%%%%%%%%%%%%%%%%%%%%%%%%%%%%%%%%%%%%%%%%%%%%%%%
\section{Graphs satisfying the inequality $k(G) \leq \dim \cH(G)+1$}
%%%%%%%%%%%%%%%%%%%%%%%%%%%%%%%%%%%%%%%%%%%%%%%%%%%%%%%%%%%%%%%%%%%%%%%%%%%%%

In this section, we show that the inequality $k(G) \leq \dim \cH(G)+1$
holds for various families of graphs. In fact, the equality holds for
a nontrivial connected triangle-free graph:

\begin{Prop}\label{trianglefree}
If $G$ is a nontrivial connected triangle-free graph,
then $k(G)=\dim \cH(G)+1$.
\end{Prop}

\begin{proof}
Since $G$ is connected, $\dim \cC(G)=|E(G)|-|V(G)|+1$.
Since $G$ is triangle-free, $\cH(G)=\cC(G)$
and so $\dim \cH(G)=|E(G)|-|V(G)|+1$.
Again, since $G$ is nontrivial, connected, and triangle-free,
$k(G)=|E(G)|-|V(G)|+2$.
Thus $k(G)=\dim \cH(G) +1$.
\end{proof}

Proposition~\ref{trianglefree} stands out in sharp contrast
to the fact that the competition number of a triangle-free graph
can be much larger than the number of its holes.
For example, the complete bipartite graph $K_{n,n}$ has
the competition number $n^2-2n+2$.
The number of holes of $K_{n,n}$ is
\[
\bi{n}{2}\bi{n}{2}=\frac{n^2(n-1)^2}{4}.
\]

Kim and Roberts~\cite{KF} showed that
if $G$ is connected and has exactly one triangle,
then $k(G)=|E(G)|-|V(G)|$ if $G$ has a hole
and $k(G)=|E(G)|-|V(G)|+1$ otherwise.
If $G$ has a hole, then $\dim \cH(G)=\dim \cC(G)$ or $\dim \cC(G)-1$.
Thus $k(G) < \dim \cH(G)+1$ if $G$ has a hole
and so the following proposition holds:

\begin{Prop}\label{onetriangle}
If a graph $G$ has exactly one triangle, then $k(G) \le \dim \cH(G)+1$.
The equality holds if and only if $G$ is a chordal graph without
isolated vertices.
\end{Prop}

Two vectors representing two distinct holes of a graph
are linearly independent.
Thus, for a graph $G$ with at most two holes, $\dim \cH(G)=h(G)$.
Since a graph $G$ with at most two holes satisfies the inequality
$k(G) \leq h(G)+1$, the following proposition holds:

\begin{Prop}\label{atmosttwoholes}
If a graph $G$ has at most two holes, then $k(G) \le \dim \cH(G)+1$.
Especially, the equality holds if $G$ is a chordal graph (has no hole)
without isolated vertices.
\end{Prop}

Furthermore, vectors representing holes any pair of which are
mutually edge-disjoint are linearly independent.
Thus, the hole dimension of any graph in the family $\cG_3$ equals
the number of its holes and so the following proposition holds:

\begin{Prop}\label{mutuallyindependent}
If $G$ is a graph such that any two distinct holes of $G$
are mutually edge-disjoint, then $k(G) \leq \dim \cH(G)+1$.
\end{Prop}

\begin{Rem}
The hole dimension of a graph can be arbitrarily larger than
its competition number.
For example, the complete multipartite graph with $m$ parts ($m\ge2$)
each of which has size $2$ has competition number $2$ (see \cite{KPS})
while $\dim \cH(G)=h(G)=\bi{m}{2}$.
\end{Rem}

In the following, we present another large class of graphs
satisfying the inequality $k(G) \leq \dim \cH(G)+1$.

\begin{Lem}\label{cycle}
Let $C$ be a cycle of a graph $G$. Then there exist cycles
$C_1, \ldots, C_t$ such that, for each $i=1, \ldots, t$,
$C_i$ is a triangle or a hole,
$V(C_i) \subseteq V(C)$, and $\chi_C=\sum_{i=1}^t \chi_{C_i}$.
\end{Lem}

\begin{proof}
We shall prove the lemma by induction on the length of a cycle.
If $C$ has length $3$, then $C$ is a triangle
and the lemma immediately follows.
Now suppose that $C$ has length at least $4$.
If $C$ does not contain a chord,
then $C$ itself is a hole and the lemma follows.
Suppose that $C$ has a chord.
Then we may take a chord $uv$ such that the length of
one of the $(u,v)$-sections of $C$ is the minimum
among the lengths of sections of $C$ determined by each of its chords.
Then the shorter $(u,v)$-section of $C$ together with $uv$ form
a triangle or a hole of $G$.
We denote it by $C_1$. Then obviously $V(C_1) \subseteq V(C)$.
The longer $(u,v)$-section of $C$ together with $uv$ form a cycle $C'$
with the length smaller than that of $C$.
By the induction hypothesis, there exist cycles $C_2$, \ldots, $C_t$
such that, for each $i=2, \ldots, t$,
$C_i$ is a triangle or hole, $V(C_i) \subseteq V(C')$,
and  $\chi_{C'}=\sum_{i=2}^t \chi_{C_i}$.
Now
\[
\chi_C = \chi_{C_1} + \chi_{C'} = \sum_{i=1}^t \chi_{C_i}.
\]
Since $V(C') \subseteq V(C)$,
$V(C_i) \subseteq V(C)$ for each $i=2, \ldots, t$.
Thus the lemma holds.
\end{proof}

\begin{Thm}\label{newfamily}
Let $G$ be a connected graph which has a connected spanning subgraph $G'$
satisfying the following properties:
\begin{itemize}
\item[{\rm (a)}]
$G'$ contains all the triangles of $G$;
\item[{\rm (b)}]
$k(G')=1$.
\end{itemize}
Then $k(G) \le \dim \cH(G)+1$.
\end{Thm}

\begin{proof}
Let $E^*:=E(G) \setminus E(G')$.
Take an edge $e \in E^*$.
Since $G'$ is connected, $G' +e$ contains a cycle containing the edge $e$.
Take a shortest cycle $C_e$ of $G'+e$ containing $e$.
From the choice, $C_e$ is a triangle or a hole in $G'+e$.
By the property (a) and the fact that $e\not\in E(G')$,
$C_e$ must be a hole in $G'+e$.
Since $C_e$ is a cycle in $G$, by Lemma~\ref{cycle},
\[
\chi_{C_e} = \sum_{C \in H^*} \chi_C + \sum_{C \in T^*} \chi_C
\]
where $H^*$ and $T^*$ are sets of holes and triangles of $G$, respectively.
We will claim that $T^*=\emptyset$ by contradiction.
Suppose that $T^* \neq \emptyset$. Let $C$ be a triangle in $T^*$.
Then one, say $e^*$, of the three edges of $C$ is a chord of $C_e$ in $G$.
By the property (a),
\[
e^* \in E(C) \subseteq E(G'),
\]
which contradicts the fact that $C_e$ is a hole in $G'+e$.
Thus $\chi_{C_e}=\sum_{C \in H^*}\chi_C$ and so $\chi_{C_e} \in \cH(G)$.
Therefore, for each edge $e\in E^*$, there exists a hole $C_e$ of $G'+e$
such that $C_e\in \cH(G)$.

Since $e' \not\in E(C_e)$ for any $e' \in E^* \setminus\{e\}$
for each $e \in E^*$, $\{\chi_{C_e} \mid e \in E^*\}$
is a linearly independent set and so $|E^*| \leq \dim \cH(G)$.
By the property (b), $k(G')=1$ and so there is an acyclic digraph $D'$
such that $C(D')=G' \cup \{i\}$
where $i$ is a new isolated vertex added to $G'$.
Now we define a digraph $D$ by
\begin{align*}
V(D) & :=V(D') \cup \{i_e \mid e \in E^*\}; \\
A(D) & :=A(D') \cup \bigcup_{e=x_ey_e  \in E^*}
\{ (x_e,i_e), (y_e,i_e) \}.
\end{align*}
It is easy to check that $D$ is acyclic
and $C(D)=G \cup \{i\} \cup \{i_e \mid e \in E^*\}$.
Hence $k(G) \le |E^*|+1 \le \dim \cH(G)+1$.
\end{proof}

Using Theorem~\ref{newfamily},
we will show that a connected graph $G$
with at most three triangles satisfies the inequality
$k(G) \le \dim \cH(G)+1$.
We need the following two lemmas.

\begin{Lem}\label{lem;forest}
Let $G$ be a connected graph.
For a forest $F$ of $G$, there exists a spanning tree $T$ of $G$
such that $E(F)\subseteq E(T)$.
\end{Lem}

\begin{proof}
We show by induction on the number of edges of a graph.
For a connected graph $G$ with $|E(G)|\le 2$,
the lemma obviously holds.
Suppose that for any connected graph with $m$ $(m\ge 1)$ edges,
the lemma is true.
Take a connected graph $G$ with $m+1$ edges.
Let $F$ be a forest of $G$.
If $G$ is a tree, then $G$ is a spanning tree satisfying 
$E(F) \subseteq E(G)$.
Suppose that $G$ is not a tree.
Then there exists a cycle $C$ of $G$.
Since $F$ does not contain a cycle,
there exists an edge $e\in E(C)$ such that $e\not\in E(F)$
and $G-e$ is connected.
Then $F$ is a forest of $G-e$ and by the induction hypothesis,
there exists a spanning tree $T$ of $G-e$ such that $E(F) \subseteq E(T)$.
Since $G-e$ is a connected spanning subgraph of $G$,
$T$ is also a spanning tree of $G$.
\end{proof}

\begin{Lem}\label{triangles}
A connected graph with at most three triangles has
a connected spanning subgraph satisfying the properties given
in Theorem~\ref{newfamily}.
\end{Lem}

\begin{proof}
Let $G$ be a connected graph with at most three triangles.
If $G$ is triangle-free, then any spanning tree satisfies
the properties given in Theorem~\ref{newfamily}.
Suppose that $G$ has at least one triangle.
Let $H$ be the subgraph of $G$ induced by the edges of triangles of $G$.
Then $H$ is chordal.
For, otherwise, there is a hole $C$ of length at least four.
Then each edge of $C$ belongs to a triangle
and any two distinct edges of $C$ cannot belong to the same triangle.
This contradicts the hypothesis that there are at most three triangles.
On the other hand, for each triangle $T$,
since there are at most two triangles other than $T$ in $G$,
there exists an edge $e_T$ not belonging to any other triangle.
Now we delete $e_T$ from $H$ for each triangle $T$ of $G$.
Then the resulting graph $F$ is a forest since $e_T$ belongs to
a unique induced cycle $T$ in $H$ for each triangle $T$.

By Lemma~\ref{lem;forest}, there exists a spanning tree $Q$ of $G$
such that $E(F) \subseteq E(Q)$.
Let $G'$ be the graph obtained by adding the edge $e_T$ to $Q$
for each triangle $T$.
Obviously $G'$ is a connected spanning subgraph of $G$ containing
all the triangles of $G$.
To show that $G'$ is chordal by contradiction,
suppose that $G'$ contains a hole $C'$.
Then at least one edge of $C'$ is $e_T$ for some triangle $T$.
Replacing $e_T$ with the other two edges of $T$ for each $e_T$ on $C'$
results in a subgraph of $Q$ since $F$ contains the two edges of $T$
other than $e_T$ for each triangle $T$.
However, each vertex of this subgraph has degree at least $2$
in this subgraph and so contains a cycle in $Q$.
This contradicts the fact that $Q$ is a tree.
Hence $k(G')=1$.
\end{proof}

By Theorem~\ref{newfamily} and Lemma~\ref{triangles},
the following theorem holds:

\begin{Thm}
If a connected graph $G$ has at most three triangles, 
then $k(G) \le \dim \cH(G)+1$.
\end{Thm}

\begin{Lem}\label{G4}
Any connected graph in $\cG_4$ has a connected spanning subgraph
satisfying the properties given in Theorem~\ref{newfamily}.
\end{Lem}

\begin{proof}
Take a graph $G \in \cG_4$. Then for every hole $C$ of $G$,
there exists $e_C \in E(G)$ such that $e_C$ is not contained
in any other induced cycle of $G$. Now let
\[
G' := G \setminus \{e_C \mid C \in H(G) \}.
\]
Since only edges are deleted, $G'$ is a spanning subgraph of $G$.
To show that $G'$ is connected by contradiction,
suppose that $G'$ is not connected.
Then $\{e_C \mid C \in H(G)\}$ contains an edge cut $F$ of $G$.
Take an $e_{C^*}=uv \in F$.
Then $(E(C^*) \setminus \{e_{C^*}\}) \cap E(G')$ contains
an edge $e^* \in F\setminus\{e_{C^*}\}$
since $u$ and $v$ belong to different components of $G'$.
Since $e^* \in F \setminus \{ e_{C^*} \} \subseteq 
\{ e_C \mid C \in H(G) \}
\setminus \{ e_{C^*} \}$,
$e^*=e_{C'}$ for a hole $C'$ distinct from $C^*$.
Then $e_{C'}$ belongs to both $C^*$ and $C'$,
which contradicts the choice of $e_{C'}$.

For each hole $C$, $e_C$ is not contained in any triangle of $G$.
Thus $G'$ contains all the triangles of $G$ and so
$G'$ satisfies the property (a).

To show that $G'$ is chordal by contradiction,
suppose that there exists a hole $C$ of $G'$.
Then $C$ is not a hole of $G$ by the definition of $G'$.
Thus $C$ contains a chord $e_{C'}=uv$ for some hole $C'$ of $G$.
Then the cycle formed by a $(u,v)$-section of $C'$ and $e_{C'}$ contains
an induced cycle of $G$ containing $e_{C'}$ by Lemma~\ref{cycle},
which contradicts the choice of $e_{C'}$.
Thus $G'$ satisfies the property (b).
\end{proof}

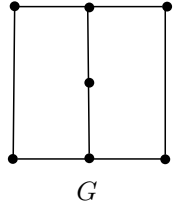
\begin{figure}[h]
\center{\scalebox{1}
{
\begin{pspicture}(0,-1.36375)(2.17,1.33375)
\psline[linewidth=0.02cm,fillcolor=black,dotsize=0.07055555cm 3.0]{**-}(0.09,1.32375)(0.07,-0.77625)
\psline[linewidth=0.02cm,fillcolor=black,dotsize=0.07055555cm 3.0]{**-}(2.07,-0.83625)(2.07,1.26375)
\psline[linewidth=0.02cm,fillcolor=black,dotsize=0.07055555cm 3.0]{**-}(2.16,1.25375)(0.14,1.25375)
\psline[linewidth=0.02cm,fillcolor=red,dotsize=0.07055555cm 3.0]{**-}(0.0,-0.76625)(2.02,-0.76625)
\psdots[dotsize=0.13](1.07,1.24375)
\psline[linewidth=0.02cm](1.05,1.26375)(1.07,-0.69625)
\psdots[dotsize=0.13](1.07,-0.75625)
\psdots[dotsize=0.13](1.07,0.24375)
\usefont{T1}{ptm}{m}{n}
\rput(1.0414063,-1.18625){$G$}
\end{pspicture}
}}
\caption{$G$ does not belong to $\cG_4$ while $G$ has
a connected spanning subgraph satisfying the properties
given in Theorem~\ref{newfamily}.}
\label{conversenottrue}
\end{figure}

%\noindent
It is worth noting that even if a graph has a connected spanning subgraph
satisfying the properties given in Theorem~\ref{newfamily},
it may not belong to $\cG_4$
(see the graph given in Figure~\ref{conversenottrue}).

By Theorem~\ref{newfamily} and Lemma~\ref{G4},
the following proposition holds:

\begin{Prop}
For any connected graph $G$ in $\cG_4$, $k(G) \le \dim \cH(G)+1$.
\end{Prop}

%%%%%%%%%%%%%%%%%%%%%%%%%%%%%%%%%%%%%%%%%%%%%%%%%%%%%%%%%%%%%%%%%%%%%%%%%%%%

\end{document}